\documentclass{amsart}

\usepackage[outer=1.25in, inner=1.25in, top=1.25in, bottom=1.25in]{geometry}
\usepackage{amsmath}
\newtheorem{theorem}{Theorem}[section]

\newtheorem{proposition}[theorem]{Proposition}
\newtheorem{lemma}[theorem]{Lemma}

\newtheorem{conjecture}[theorem]{Conjecture}

\newtheorem{corollary}[theorem]{Corollary}

\DeclareMathOperator{\sub}{Sub}
\DeclareMathOperator{\syl}{Syl}
\newcommand{\binomq}{\genfrac{[}{]}{0pt}{}}
\usepackage{hyperref,amssymb,xcolor}
\begin{document}
	
	\title[Number of subgroups]{Counting subgroups of a finite group containing a prescribed subgroup}
	
	\author[L. Guerra]{Lorenzo Guerra}
	\address{Lorenzo Guerra, Dipartimento di Matematica Pura e Applicata,\newline
		University of Milano-Bicocca, Via Cozzi 55, 20126 Milano, Italy} 
	\email{l.guerra@unimib.it}

	\author[F. Mastrogiacomo]{Fabio Mastrogiacomo}
	\address{Fabio Mastrogiacomo, Dipartimento di Matematica ``Felice Casorati", University of Pavia, Via Ferrata 5, 27100 Pavia, Italy} 
	\email{fabio.mastrogiacomo01@universitadipavia.it}

	\author[P. Spiga]{Pablo Spiga}
	\address{Pablo Spiga, Dipartimento di Matematica Pura e Applicata,\newline
		University of Milano-Bicocca, Via Cozzi 55, 20126 Milano, Italy} 
	\email{pablo.spiga@unimib.it}
	\thanks{The authors are members of the GNSAGA INdAM research group and kindly acknowledge their support.\\ The first and third authors are funded by the European Union via the Next
		Generation EU (Mission 4 Component 1 CUP B53D23009410006, PRIN 2022, 2022PSTWLB, Group
		Theory and Applications)}
	\maketitle	
	\begin{abstract}  
		Let $R$ be a finite group, and let $T$ be a subgroup of $R$. We show that there are at most
		\[
			7.3722[R:T]^{\frac{\log_2[R:T]}{4}+1.8919}
		\]
		subgroups of $R$ containing $T$.
	\end{abstract}
	
	\section{introduction}
Let $R$ be a finite group of order $r$. Since any strictly descending chain of subgroups in $R$ has length at most $\log_2 r$, it follows that every subgroup $H \leq R$ can be generated by at most $\lfloor \log_2 r \rfloor$ elements. Consequently, the total number of subgroups in $R$ is bounded above by
\[
r^{\lfloor \log_2 r \rfloor} \leq r^{\log_2 r}.
\]
This classical estimate gives a general upper bound for the number of subgroups in a finite group.

This bound turns out to be nearly tight. For example, if $R$ is an elementary abelian $2$-group of order $r = 2^a$, then the number of subgroups of order $2^{\lfloor a/2 \rfloor}$ is given by the $2$-binomial coefficient
\[
\binomq{a}{\lfloor \tfrac{a}{2} \rfloor}_2 =
\frac{2^a - 1}{2^{\lfloor \tfrac{a}{2} \rfloor} - 1}
\cdot \frac{2^{a-1} - 1}{2^{\lfloor \tfrac{a}{2} \rfloor - 1} - 1}
\cdots
\frac{2^{a - \lfloor \tfrac{a}{2} \rfloor + 1} - 1}{2 - 1}.
\]
This quantity satisfies $\binomq{a}{\lfloor a/2 \rfloor}_2 \geq 2^{a^2/4}$, so in this case we have
\[
2^{a^2/4} = r^{\log_2 r / 4}
\]
as a lower bound for the number of subgroups of $R$.

A much stronger result due to Borovik, Pyber, and Shalev~\cite[Corollary~1.6]{BPS}, which relies on the Classification of Finite Simple Groups, shows that for any finite group of order $r$, the number of subgroups is at most
\[
r^{\log_2 r \cdot \left(\tfrac{1}{4} + o(1)\right)}.
\]
This confirms that the example above essentially achieves the maximal possible growth rate.

In many applications  having an explicit constant in the exponent is important. In this direction, Spiga~\cite{Spiga} obtained an effective version of the result: every finite group of order $r$ has at most
\[
7.3722 \cdot r^{\frac{\log_2 r}{4} + 1.5315}
\]
subgroups. The bound was subsequently improved in~\cite{FusSp}, where the additive constant $1.5315$ in the exponent was removed. This explicit bound was also used as a key step in the proof of the main theorem of~\cite{GSX}.

	In this paper, we are interested in a generalization of this problem. Let $R$ be a finite group, and let $T \leq R$ be a subgroup. Define
	\[
	\sub(R,T) = \{ H \leq R \, : \, T \leq H \}.
	\]
	Our goal is to establish an upper bound for $|\sub(R,T)|$. Using the same argument as above, we can obtain a preliminary bound: if $|R| = r$ and $|T| = t$, then any chain of subgroups of $R$ containing $T$ has length at most $\log_2(r/t)$. Therefore, each such subgroup can be generated by a set of cardinality at most $\lfloor \log_2(r/t) \rfloor$, together with $T$. Thus, there are at most
	\[
	\left( \frac{r}{t} \right)^{\log_2 \left( \frac{r}{t} \right)}
	\]
	subgroups of $R$ containing $T$. As noted earlier, this bound is not optimal—this is already evident when $T$ is the trivial subgroup.
	
	We present here a refined version of this inequality, generalizing the main result of~\cite{Spiga}.
	
	\begin{theorem}\label{thm:main}
		Let $R$ be a finite group, and let $T \leq R$. Then
		\[
		|\sub(R,T)| \leq 7.3722 \cdot [R:T]^{\frac{\log_2 [R:T]}{4} + 1.8919}.
		\]
	\end{theorem}
	
	We believe that the constant $1.8919$ in the exponent can be removed, as in the main result of~\cite{FusSp}.
	
	This theorem also admits a natural interpretation in the context of permutation groups. Suppose that $R$ is a finite transitive group acting on a set $\Omega$, with point stabilizer $T$ and degree $n = [R:T]$. It is well known that the systems of imprimitivity of $R$ correspond bijectively to the subgroups of $R$ containing the stabilizer. Thus, Theorem~\ref{thm:main} yields the following corollary.	
	\begin{corollary}
		Let $R$ be a transitive group of degree $n$. Then $R$ has at most
		\[
		7.3722 \cdot n^{\frac{\log_2 n}{4} + 1.8919}
		\]
		systems of imprimitivity.
	\end{corollary}
	
	The proof of Theorem~\ref{thm:main} closely follows the argument in~\cite{Spiga}, which gives an elementary proof of~\cite[Corollary~$1.6$]{BPS}. A central component of the argument in~\cite[Corollary~$1.6$]{BPS} relies on estimating the number of maximal subgroups in a finite group. The papers~\cite{B, LPS} provide valuable results in this direction, though the presence of implicit constants makes it difficult to extract explicit bounds.\\
	The analogous result for the number of maximal system of imprimitivity can be found in~\cite{LMS}.\\

	For a non-zero integer $n$, we define $\lambda(n)$ as the number of prime divisors of $n$, counted with multiplicities. We believe that our main result can be improved by replacing $\log_2 ([R:T])$ with $\lambda([R:T])$.
	\begin{conjecture}
		There exists a constant $c>0$ such that if $R$ is a finite group, and $T \leq R$, then
		\[
			|\sub(R,T)| \leq [R:T]^{\lambda([R:T])+c}.
		\]
	\end{conjecture} 
	We have little evidence for this conjecture: it is true for nilpotent groups, from the results in Section~\ref{sec:1}.
	The function $\lambda$ is inspired by its analogous use in the famous and important result of Pyber in~\cite{Py}.\\

	A key step in the proof is establishing Theorem~\ref{thm:main} for $p$-groups. This is accomplished in Section~\ref{sec:1}. The general case is addressed in Section~\ref{sec:2}, where we first present the subgroup-counting argument and then develop some arithmetic observations that lead to the final bound.
	
	Throughout, if $k$ is an integer, $R$ is a group and $T \leq R$, we define
	\[
	\sub_k(R,T) = \{ H \in \sub(R,T) \, : \, [R:H] = k \}.
	\]
	\section{Bound for groups of prime power order}\label{sec:1}
	If $p$ is a prime and $m,r$ are non-negative integers, the Gaussian binomial coefficient is defined by
	\[
		\binomq{m}{r}_p = \prod_{i=0}^{r-1}\frac{p^{m-i}-1}{p^{i+1}-1}.
	\]
	We will use this recurrence equation involving Gaussian binomial coefficient:
	\[
		\binomq{m+1}{r}_p = \binomq{m}{r}_p p^r + \binomq{m}{r-1}_p.
	\]
	\begin{proposition}\label{prop:pgrp}
		Let $P$ be a $p$-group, with $p$ prime, let $T$ be a subgroup of $R$ having index $p^c$. Then,
		\[
			|\sub_{p^k}(P,T)| \leq \binomq{c}{k}_p.
		\]
	\end{proposition}
	\begin{proof}
		We argue by induction on $k$. If $k=0$, then the statement is trivial. Suppose that $k>0$, and take a maximal subgroup $M \leq P$ containing $T$. There are two possible cases for a subgroup $H$ of index $p^k$: either $H \leq M$, or $H \not \leq M$.\\
		If $H \leq M$,  then by induction there are at most
		\[
			\binomq{c-1}{k-1}_p
		\]
		choices for $H$. \\
		Suppose that $H \not \leq M$. Take $x \in P \setminus M$. Then, $P = \langle M,x \rangle$. Moreover, since $M \trianglelefteq P$, every element of $P$ is of the form $mx^\alpha$, for some $\alpha \in \mathbb{Z}$ and for some $m \in M$. Being $M$ maximal, $x^p \in M$. Therefore, up to the choice of $m$, every element of $P$ is of the form $m' x^\beta$, for some $m' \in M$, and $\beta \in \{0,1,\dots,p-1\}$. \\
		Observe that $H = \langle H \cap M, mx^\beta \rangle$, with $m \in M$ and $\beta \in \{1,\dots,p-1\}$. Up to replacing $mx^\beta$ with $(mx^\beta)^\gamma$, where $\beta \gamma \equiv 1 \mod p$, we can assume, without loss of generality, that $\beta =1$, so that $H = \langle H \cap M, mx\rangle$. 
		Moreover, if $(H \cap M) m' = (H \cap M)m$, then $\langle H \cap M, mx\rangle = \langle H \cap M, m'x\rangle$. Therefore, to determine the subgroup $H$ it is enough to specify its intersection with $M$, and a coset in $M$ of $H \cap M$. There are $[M : H \cap M] = p^k$ cosets, and, by induction, at most 
		\[
			\binomq{c-1}{k}_p
		\]
		choices for $H \cap M$. Thus, we have in total at most
		\[
			\binomq{c-1}{k-1}_p + \binomq{c-1}{k}_p p^k = \binomq{c}{k}_p
		\]
		subgroups of index $p^k$.
	\end{proof}
	In~\cite{Spiga}, the author defines a function $S(p,a)$ such that each group of order $p^a$ has at most $S(p,a)$ subgroups. The function is defined as follows.
	\[
		S(p,a) = 
		\begin{cases}
			2 &\text{ if } a=1, \\
			p+3 &\text{ if } a=2,\\
			2p^2+2p+4 &\text{ if } a=3,\\
			p^4+3p^3+4p^2+3p+5 &\text{ if } a=4,\\
			2p^6 + 2p^5+6p^4+6p^3+6p^2+4p+6 &\text{ if } a=5,\\
			c(p)p^{\frac{a^2}{2}} &\text{ if } a\geq 6,
		\end{cases}
	\]
	where 
	\begin{align}\label{eq:cp}
		&C(p) = \prod_{i\geq1} \frac{1}{1-\frac{1}{p^i}}, \\
		&c(p) = 2.129 C(p).
	\end{align}
	Moreover, he observed that
	\[
		\sum_{k=0}^{a}\binomq{a}{k}_p \leq S(p,a).
	\]
	In particular, if $P$ and $T$ are as in Proposition~\ref{prop:pgrp}, we have
	\[
		|\sub(P,T)| \leq \sum_{k=0}^{c} |\sub_{p^k}(P,T)| \leq \sum_{k=0}^{c}\binomq{c}{k}_p \leq S(p,c).
	\]
	\section{Proof of Theorem~\ref{thm:main}}\label{sec:2}
	We start with a preliminary observation.
\begin{lemma}\label{l:1}
Let $R$ be a finite group, let $T$ be a subgroup of $R$, let $p$ be a prime number, and let $\syl_p(R,T)$ denote the set of $p$-Sylow subgroups $P$ of $R$ such that $T \cap P \in \syl_p(T)$. Then the number of $T$-orbits under conjugation on $\syl_p(R,T)$ is at most $[R:T]/p^c$, where $p^c$ is the largest power of $p$ dividing $[R:T]$.
\end{lemma}

\begin{proof}
Let $P \in \syl_p(R,T)$. By Sylow's theorem, each element of $\syl_p(R,T)$ is of the form $P^x$ for some $x \in R$ such that $T \cap P^x \in \syl_p(T)$. Define
\[
\mathcal{P} = \{x \in R \mid T \cap P^x \in \syl_p(T)\},
\]
and observe that $\mathcal{P}$ is a union of $({\bf N}_R(P), T)$-double cosets.

Let $x_1, \ldots, x_t$ be a set of representatives for the $({\bf N}_R(P), T)$-double cosets in $\mathcal{P}$. Observe that $t$ is the number of $T$-orbits on $\syl_p(R,T)$. Then,
\begin{align*}
\frac{|{\bf N}_R(P)x_iT|}{|T|} &= |{\bf N}_R(P) : {\bf N}_R(P) \cap T^{x_i^{-1}}| = |{\bf N}_R(P^{x_i}) : {\bf N}_R(P^{x_i}) \cap T|.
\end{align*}
Since the largest power of $p$ dividing $|{\bf N}_R(P^{x_i})|$ is $|P|$, and the largest power of $p$ dividing $|{\bf N}_R(P^{x_i}) \cap T|$ is $|T \cap P|$, it follows that
\[
|{\bf N}_R(P^{x_i}) : {\bf N}_R(P^{x_i}) \cap T|
\]
is divisible by $p^c$. Therefore, each $({\bf N}_R(P), T)$-double coset in $\mathcal{P}$ consists of at least $p^c$ distinct right cosets of $T$ in $R$. Consequently, the number of such double cosets is at most $[R:T]/p^c$.
\end{proof}
	
	\begin{proposition}\label{prop:nrsbgrpindex}
		Let $R$ be a finite group, and let $T < R$. Suppose that 
		\[
			[R:T]=\prod_{i=1}^{\ell}p_i^{c_i},
		\]
		where $c_i>0$ and $p_i$ are distinct primes. Then,
		\[
			|\sub(R,T)| \leq [R:T]^{\ell-1} \prod_{i=1}^{\ell} S(p_i,c_i).
		\]
	\end{proposition}
	\begin{proof}

Let $H \in \sub(R,T)$. Then $H$ is generated by $T$ and an $\ell$-tuple $(H_1, \ldots, H_\ell)$, where each $H_i$ is a $p_i$-Sylow subgroup of $H$.

As in Lemma~\ref{l:1}, for each $i \in \{1, \ldots, \ell\}$, let
\[
\syl_{p_i}(R,T) = \{ R_i \in \syl_{p_i}(R) \mid R_i \cap T \in \syl_{p_i}(T) \}.
\]
Possibly replacing the $\ell$-tuple $(H_1, \ldots, H_\ell)$, we may assume without loss of generality that $H_i \in \syl_{p_i}(H,T)$ for all $i$. By Sylow's theorem, each $H_i$ is contained in a $p_i$-Sylow subgroup $R_i$ of $R$. Observe that
\[
R_i \cap T = R_i \cap H \cap T = H_i \cap T \in \syl_{p_i}(T),
\]
so $R_i \in \syl_{p_i}(R,T)$. Moreover, since $H$ contains $T$, for every $t_1, \dots, t_{\ell}\in T$, the $\ell$-tuples $(H_1,\ldots,H_\ell)$ and $(H_1^{t_1},\ldots,H_\ell^{t_\ell})$ generate together with $T$ the same subgroup $H$.

Therefore, to bound the number of subgroups of $R$ containing $T$, we estimate the number of $\ell$-tuples $(R_1, \ldots, R_\ell)$ with $R_i \in \syl_{p_i}(R,T)$, taken up to conjugation by $T$. Once the $R_i$ are fixed, we count the number of possible subgroups $H_i \leq R_i$ such that $H_i \in \syl_{p_i}(H,T)$.

By Proposition~\ref{prop:pgrp}, we know that
\[
|\sub(R_i, R_i \cap T)| \leq S(p_i, c_i).
\]
Applying Lemma~\ref{l:1}, the number of $T$-conjugacy classes of such $R_i$ is at most $[R:T]/p_i^{c_i}$.
Hence, the number of such $\ell$-tuples is at most
\[
\prod_{i=1}^{\ell} \frac{[R:T]}{p_i^{c_i}} = [R:T]^{\ell - 1}
\]
and the result follows.
	\end{proof}

	\subsection{Arithmetic observations }
	\begin{lemma}\label{lemma0}Let $p$ be a prime number, let $c$ be a positive integer and let $r/t$ be a multiple of $p^{c}$. Then $S(p,c)\le 7.3722\cdot p^{c\frac{\log_2 (r/t)}{4}}$.
	\end{lemma}
	\begin{proof}
		Since $r/t\to\log_2 (r/t)$ is monotone increasing, we may suppose that $r/t=p^{c}$. In particular, $7.3722\cdot p^{c\frac{\log_2r/t}{4}}=7.3722
		\cdot p^{\frac{c^2}{4}\log_2p}$. Now the proof follows by distinguishing various possibilities for $c$. When $c=1$, we have $S(p,1)=2$ and $$7.3722\cdot p^{\frac{c^2}{4}\log_2p}=7.3722\cdot p^{\frac{\log_2p}{4}}\ge 7.3722\cdot 2^{1/4}=9.5136.$$ When $c=2$, $S(p,2)=p+3$ and $$7.3722\cdot p^{\frac{c^2}{4}\log_2p}=7.3722\cdot p^{\log_2p}\ge 7.3722\cdot p;$$ clearly, $p+3\le 7.3722\cdot p$. The cases $c\in \{3,4,5\}$ are entirely similar.
		
		Suppose $c\ge 6$. Now, $c(p)p^{\frac{c^2}{4}}=S(p,c)\le 7.3722\cdot p^{\frac{c^2}{4}}$ if and only if $c(p)\le 7.3722$. From~\eqref{eq:cp}, $p\mapsto c(p)$ is a monotone decreasing function and hence $c(p)\le c(2)=7.372187<7.3722$.
	\end{proof}
	
	\begin{lemma}\label{lemma:numeric}
		Let $r$ be an integer, let $t$ be a divisor of $r$,  let $p$ be a prime divisor of $r/t$, and let $p^c$ be the largest power of $p$ dividing $r/t$. Then,
		\[
			\frac{r}{t}S(p,c)\leq \left(\frac{r}{t}\right)^{c\frac{\log_2 \left(\frac{r}{t}\right)}{4}},
		\]
		unless $(p,c,r/t)$ are as in Table~$\ref{Table:expcetions}$.
	\end{lemma}
	\begin{proof}
		It follows from Lemmas~$3.3$-$3.8$ of~\cite{Spiga}.
	\end{proof}
	\begin{table}[]
		\begin{tabular}{|c|c|c|}
			\hline
			$c$ & $p$ & Comments on $r/t$   \\
			\hline
			$1$& $23$     &   $r/t\leq 184$     \\
			 & $19$     &   $r/t\leq 71\,896$   \\
			 & $2,3,5,7,11,13,17$ & any $r/t$\\
			 \hline
			$2$ & $7$& $r/t \leq 294$ \\
			 & $5$ & $  r/t \leq 407\,850$ \\
			 & $2,3$ & any $r/t$ \\
			 \hline
			 $3$ & $5$ & $r/t\leq 250$ \\
			 	& $2,3$ & any $r/t$ \\
			 \hline 
			 $4$ & $3$ & $r/t\leq 9\,396$ \\
			     & $2$   &  any $r/t$\\
			 \hline 
			 $5$ & $3$ & $r/t\leq 34\,375$ \\
			 	 & $2$ & any $r/t$\\
			\hline	
			$c\geq 6$ & $3$ & $r/t \in \{729,1\,458,2\,187,2\,916\}$ \\
						& $2$ & any $r/t$\\
			\hline
		\end{tabular}
		\caption{Exceptions in Lemma~\ref{prop:exceptions} with small primes}
	\end{table}\label{Table:expcetions}
	
	\begin{proposition}\label{prop:exceptions}
		Let $R$ be a finite group, let $T $ be a subgroup of $R$, let $p$ be a prime divisor of $[R:T]$, and let $p^c$ be the largest power of $p$ dividing $[R:T]$. Then, either $r/t S(p,c) \leq p^{c\frac{\log_2 (r/t)}{4}}$, or one of the following holds
		\begin{enumerate}
			\item\label{corollary1:1} $R$ and $T$ satisfy Theorem~\ref{thm:main},
			\item\label{corollary1:2} $p \in \{5,7,11,13,17\}$ and $c=1$,
			\item\label{corollary1:3} $p = 3$ and $c \leq 3$,
			\item\label{corollary1:4} $p=2$.
		\end{enumerate}
	\end{proposition}
	\begin{proof}
		Suppose that $(r/t)S(p,c)>p^{c\frac{\log_2(r/t)}{4}}$. In particular, $(c,p,r/t)$ are as in Table~\ref{Table:expcetions}. We consider in turn each of these cases.
		
		When $c=1$, we are as in the first row of Table~\ref{Table:expcetions}. In particular, we only need to deal with $p=19$ and $p=23$, because when $p\le 17$ we see that parts~\eqref{corollary1:2}--\eqref{corollary1:4} are satisfied. In particular, we only have a finite number of cases to consider. Let $r=p_1^{a_1}\cdots p_\ell^{a_\ell}$ and $t = p_1^{b_1}\cdots p_\ell^{b_\ell}$ be the factorization of $r$ and $t$,  into distinct prime powers. Let us consider the function
		$$f\left(\frac{r}{t}\right):=\left(\frac{r}{t}\right)^{\ell-1}\prod_{i=1}^\ell S(p_i,a_i-b_i).$$
		By Proposition~\ref{prop:nrsbgrpindex}, if $f(r/t)\le 7.3722\cdot (r/t)^{\log_2(r/t)/4+1.8919}$, then Theorem~\ref{thm:main} holds and hence part~\eqref{corollary1:1} is satisfied. Therefore, we may suppose that $f(r/t)> (r/t)^{\log_2(r/t)/4+1.8919}$.
		We have implemented this function in a computer and we have checked that no $r/t$ in the range described in  Table~\ref{Table:expcetions} satisfies $f(r/t)>7.3722\cdot (r/t)^{\log_2(r/t)/4+1.8919}$.
		
		When $c=2$, we are as in the second row of Table~\ref{Table:expcetions}. In particular, we only need to deal with $p=5$ and $p=7$, because when $p\le 3$ we see that parts~\eqref{corollary1:2}--\eqref{corollary1:4} are satisfied. In particular, we only have a finite number of cases to consider. We have checked that no $r/t$ in the range described in Table~\ref{Table:expcetions} satisfies $f(r/t)>7.3722\cdot (r/t)^{\log_2(r/t)/4+1.8919}$. 
		
		The  the cases $a\ge 3$ are analogous.
	\end{proof}
	We are now ready to complete the proof of Theorem~\ref{thm:main}.  As in Proposition~\ref{prop:nrsbgrpindex}, we write $[R:T]=\prod_{i=1}^\ell p_i^{c_i}$. Let $\mathcal{I}$ be the collection of indices $i \in \{1,2,\dots,\ell\}$ with $$\frac{r}{t}S(p_i,c_i) >  p_i^{c_i\frac{\log_2 (r/t)}{4}},$$ and let $\mathcal{P} = \{p_i \, : \, i \in \mathcal{I}\}$. From Proposition~\ref{prop:exceptions}, we can restrict to the case where $\mathcal{P}\subseteq \{2,3,5,7,11,13,17\}$. Moreover, if $i \in \mathcal{I}$ with $p_i \in \{5,7,11,13,17\}$, then $c_i = 1$. In particular, $|\mathcal{I}|\le 7$.
	
From Propositions~\ref{prop:nrsbgrpindex} and~\ref{prop:exceptions}, we have
\begin{align}\label{eq:1}
|\mathrm{Sub}(R,T)|&\le (r/t)^{|\mathcal{I}|-1}\prod_{i\in I}S(p_i,c_i)\prod_{\substack{i=1\\i\notin \mathcal{I}}}^\ell (r/t)S(p_i,c_i)
\le (r/t)^{|\mathcal{I}|-1}\prod_{i\in \mathcal{I}}S(p_i,c_i)\prod_{\substack{i=1\\i\notin \mathcal{I}}}^\ell p_i^{\frac{c_i\log_2(r/t)}{4}}.
\end{align}

Now, when $2\in\mathcal{P}$, let $i_0\in \mathcal{I}$ with $p_{i_0}=2$ and set $\mathcal{I}_0=\{i_0\}$; when $2\notin\mathcal{P}$, set $\mathcal{I}_0=\emptyset$. Let $\alpha>0$ be  minimal subject to
$$\left(\frac{r}{t}\right)^{|\mathcal{I}|-|\mathcal{I}|_0}\prod_{i\in\mathcal{I}\setminus\mathcal{I}_0}S(p_i,c_i)\le \left(\frac{r}{t}\right)^\alpha \prod_{i\in\mathcal{I}\setminus\mathcal{I}_0}p_i^{c_i\frac{\log_2(r/t)}{4}}.$$
By taking the logarithm on both sides of this inequality and rearranging the terms we deduce
\begin{align*}
\alpha&\ge |\mathcal{I}|-|\mathcal{I}_0|+\frac{\sum_{i\in\mathcal{I}\setminus\mathcal{I}_0}\log(S(p_i,c_i))}{\log(r/t)}-\sum_{i\in\mathcal{I}\setminus\mathcal{I}_0}\frac{c_i\log(p_i)}{4\log(2)}.
\end{align*}
Observe now that $r/t\ge \prod_{i\in\mathcal{I}\setminus\mathcal{I}_0}p_i^{c_i}$ and hence
\begin{align}\label{eq:26}
\alpha&\ge |\mathcal{I}|-|\mathcal{I}_0|+\frac{\sum_{i\in\mathcal{I}\setminus\mathcal{I}_0}\log(S(p_i,c_i))}{\sum_{i\in\mathcal{I}\setminus\mathcal{I}_0}c_i\log(p_i)}-\sum_{i\in\mathcal{I}\setminus\mathcal{I}_0}\frac{c_i\log(p_i)}{4\log(2)}.
\end{align}
Recall $\{p_i\mid i\in\mathcal{I}\setminus\mathcal{I}_0\}\subseteq\{3,5,7,11,13,17\}$ and recall that $c_i\le 3$ for each $i$ (in fact, $c_i=1$ for each $i$, except possibly for the prime $3$, where $c_i\le 3$). Therefore, the right hand side of~\eqref{eq:26} can be explicitly computed with the help of a calculator. We obtain that we may choose $\alpha=1.8919$.

Now, from~\eqref{eq:1}, when $\mathcal{I}_0=\emptyset$, we deduce
\begin{align*}
|\mathrm{Sub}(R,T)|&\le 
 (r/t)^{\alpha-1}\prod_{i=1}^\ell p_i^{\frac{c_i\log_2(r/t)}{4}}=(r/t)^{\log_2(r/t)/4+\alpha-1}.
\end{align*}	
Similarly, when $\mathcal{I}_0\ne\emptyset$, from Lemma~\ref{lemma0} applied with $p=2$, we deduce
\begin{align*}
|\mathrm{Sub}(R,T)|&\le 
 S(2,c_{i_0})(r/t)^{\alpha}\prod_{\substack{i=1\\i\ne i_0}}^\ell p_i^{\frac{c_i\log_2(r/t)}{4}}\le 7.3722 (r/t)^\alpha
 \prod_{i=1}^\ell p_i^{\frac{c_i\log_2(r/t)}{4}}=
 7.3722 (r/t)^{\log_2(r/t)/4+\alpha}.
\end{align*}


\thebibliography{10}
	\bibitem{B}A.~V.~Borovik, On the Number of Maximal Soluble
	Subgroups of a Finite Group, \textit{Comm. Algebra} \textbf{26}, 4041--4050.
	\bibitem{BPS}A.~V.~Borovik, L.~Pyber, A.~Shalev, Maximal subgroups in finite and profinite groups,  \textit{Trans. Amer. Math. Soc.} \textbf{348} (1996), 3745--3761.
	\bibitem{FusSp} M.~Fusari, P.~Spiga, On the maximum number of subgroups of a finite group, \textit{Journal of Algebra}, \textbf{635} (2023), 486--526.
	\bibitem{GSX}	Y.~Gan, P.~Spiga, and B.~Xia,
	Asymptotic enumeration of Haar graphical representations,
	\textit{arXiv preprint} \textbf{arXiv:2405.09192}, 2024.
	\bibitem{LPS}M.~W.~Liebeck, L.~Pyber, A.~Shalev, On a conjecture of G.~E.~Wall, \textit{J. Algebra} \textbf{317} (2007), 184--197.
	\bibitem{LMS} A.~Lucchini, M.~Moscatiello, and P.~Spiga,
	A polynomial bound for the number of maximal systems of imprimitivity of a finite transitive permutation group,
	\textit{Forum Math.} \textbf{32} (2020), no.~3, 713--721.
	\bibitem{Py} L.~Pyber, Enumerating Finite Groups of Given Order. \textit{Annals of Mathematics}, 137, (1993), 203–220.
	\bibitem{Spiga} P.~Spiga, An explicit upper bound on the number of subgroups of a finite group, \textit{Journal of Pure and Applied Algebra}, \textbf{227} (2023), no.~6, 107312.
\end{document}